\documentclass[11pt]{amsart}

\usepackage{amssymb,amsmath,amscd,graphicx,
latexsym,amsthm,yhmath}
\usepackage{amssymb,latexsym,amsmath,amscd,graphicx}
\usepackage[mathscr]{eucal}
  \usepackage[all]{xy}
\def\OO{{\mathcal O}}

\def\F{\mathcal{F}}

\def\G{\mathcal{G}}
\def\H{\mathcal{H}}

\def\Pic0{\mathrm{Pic}^0}
\def\Alb{\mathrm{Alb}\,}

\def\(id){\,\mathrm{id}}
\def\reg{\mathrm{reg}}

\theoremstyle{plain}

\newtheorem{theorem}{Theorem}[section]
\newtheorem{theoremalpha}{Theorem}
\newtheorem{corollaryalpha}[theoremalpha]{Corollary}

\newtheorem{proposition/example}[theorem]{Proposition/Example}
\newtheorem{proposition}[theorem]{Proposition}

\newtheorem{lemma}[theorem]{Lemma}

\theoremstyle{definition}
\newtheorem{definition}[theorem]{Definition}

\newtheorem{remark}[theorem]{Remark}

\newtheorem{conjecture/question}[theorem]{Conjecture/Question}

\newtheorem{remark/definition}[theorem]{Remark/Definition}
\newtheorem{notation/assumptions}[theorem]{Assumptions/Notation}

\numberwithin{equation}{section}

 \theoremstyle{remark}

\begin{document}

\title{derived invariants 
 arising from the Albanese map}

\author{Federico Caucci}
  \address{Sapienza Universit\`a di Roma, P.le Aldo Moro 5, I-00185 Roma, Italy}
 \email{{\tt caucci@mat.uniroma1.it}}
 \thanks{}

\author{Giuseppe Pareschi}
  \address{ Universit\`a di Roma Tor Vergata, V.le della Ricerca Scientifica, I-00133 Roma, Italy}
 \email{{\tt pareschi@mat.uniroma2.it}}

\maketitle

\setlength{\parskip}{.1 in}
\begin{abstract} Let $a_X:X\rightarrow \mathrm{Alb}\, X$ be the Albanese map of a  smooth  complex projective  variety.  
  Roughly speaking  in this note we prove that for all $i \geq 0$ and   $\alpha\in \mathrm{Pic}^0\, X$, the cohomology ranks 
  $h^i(\mathrm{Alb}\, X, \,{a_X}_* \omega_X\otimes P_\alpha)$ 
 are derived invariants. In the case of  varieties of maximal Albanese dimension this proves 
  conjectures of Popa and
 Lombardi-Popa -- including the derived invariance of the Hodge numbers $h^{0,j}$ --   and  a weaker version of them for arbitrary varieties. Finally we provide  an application to derived invariance of certain irregular fibrations.
\end{abstract}

\section{introduction} 

This paper is about derived invariants of  smooth complex  projective varieties (henceforth called varieties) arising from the Albanese morphism 
\[a_X :X\rightarrow \Alb X.\] 
For example, a fundamental result of Popa and Schnell shows that the dimension of the Albanese variety $ \Alb X$, i.e.  $h^0(\Omega^1_X)$, is a derived invariant (\cite{ps1}).  

Roughly speaking in this note we prove 
  that,  for all $i\ge 0$ and for all $\alpha\in\Pic0 X$ the cohomology ranks 
  \[h^i(\Alb X, \,{a_X}_* \omega_X\otimes P_\alpha)\] 
 are derived invariants.  In the case of varieties of maximal Albanese dimension\footnote{this means that $\dim a_X(X)=\dim X$. To be of maximal Albanese dimension is a derived invariant property \cite{lombardi-derived}} this settles in the affirmative (a strengthened version of) a conjecture of 
 Lombardi and Popa  (\cite{popa-conj2} Conjecture 11) -- proved by Lombardi for $i=0$ and partially for $i=1$ (\cite{lombardi-derived}) -- and proves a weaker version of it for arbitrary varieties. For varieties of maximal Albanese dimension this implies   the derived invariance of the Hodge numbers $h^{0,j}$ for all $j\ge 0$  and of all canonical cohomological support loci, proving in this case another conjecture of Popa. 
   In this direction previous results in low dimension were obtained by Popa, Lombardi and Abuaf (\cite{popa-conj1}, \cite{lombardi-derived}, \cite{popa-conj2}, \cite{abuaf}). 
 
  Turning to  precise statements, for a variety $X$ let us denote $\mathbf D(X)$ its bounded derived category of coherent sheaves. Let $Y$ be another 
  variety and let
\[\varphi:\mathbf D(X)\rightarrow \mathbf D(Y)\]
 be an exact equivalence. 
  As shown by R. Rouquier (\cite{rouquier}, see also \cite{ps1}), $\varphi$ induces an isomorphism of algebraic groups
\begin{equation}\label{rou}\overline\varphi: \mathrm{Aut}^0\, X \times \Pic0 X \rightarrow \mathrm{Aut}^0\, Y \times \Pic0 Y. 
\end{equation}
  We choose normalized Poincar\'e line bundles 
   so that to a closed point $\alpha\in \Pic0 X$ (resp. $\beta \in \Pic0 Y$) corresponds the line bundle $P_\alpha$ on $X$ (resp. $P_\beta$ on $Y$). Essential for our arguments is a result of Lombardi, from which it follows  that if $h^i(\Alb X,{a_X}_*\omega_X\otimes P_\alpha)>0$ for some  $i\ge 0$ then $\overline\varphi(\(id)_X,P_\alpha)$ is of the form $(\(id)_Y,P_\beta)$ for a $\beta\in\Pic0 Y$.  If this is the case we will abusively denote
\[\beta=\overline\varphi(\alpha)\]
\begin{theoremalpha}\label{main} Let $i\in\mathbb N$. In the above notation, $h^i(\Alb X, \, {a_X}_*\omega_X\otimes P_\alpha)>0$ if and only if $h^i(\Alb Y, \, {a_Y}_*\omega_Y\otimes P_{\bar\varphi(\alpha)})>0$. If this is the case
\[h^i(\Alb X, \, {a_X}_*\omega_X\otimes P_\alpha)=h^i(\Alb Y, \,  {a_Y}_*\omega_Y\otimes P_{\bar\varphi(\alpha)}).\]
\end{theoremalpha}
 
 It is  expected that derived-equivalent varieties have the same Hodge numbers. As a consequence of the invariance
 of the cohomological ranks of the sheaves ${a_X}_*\omega_X$ it follows that this holds true for the $h^{0,j}$'s of varieties of maximal Albanese dimension.
   \begin{corollaryalpha}\label{cor1} Let $X$ and $Y$ be smooth complex projective varieties with equivalent derived categories. 
  Then, for all $i\in\mathbb N$
  \[h^i(\Alb X,{a_X}_*\omega_X)= h^i(\Alb Y,{a_Y}_*\omega_Y).\]
  In particular, if $X$ is of maximal Albanese dimension then, for all $j\ge 0$,
  \[h^{0,j}(X)=h^{0,j}(Y).\]
  \end{corollaryalpha} 
  
    Notice that in the maximal Albanese dimension case  $R^i{a_X}_*\omega_X = 0$   for  $i>0$  (Grauert-Riemenschneider vanishing) and therefore $h^i(X, \omega_X)= h^i(\Alb X, {a_X}_*\omega_X)$. This proves the last part of the Corollary. 
    
    Given a coherent sheaf $\F$ on a smooth projective variety $X$ its \emph{cohomological support loci} are the following algebraic subvarieties of $\Pic0 X$: 
  \[V^i_{ r}(X,\F)=\{\alpha\in \Pic0 X\>|\> h^i(X, \F\otimes P_\alpha)\ge r\}.\]
  For $r=1$ we simply denote $V^i_{ r}(X,\F)=V^i(X,\F)$. 
  Again, by Grauert-Riemenschneider and projection formula it follows that $V^i_r(X,\omega_X)=V_r^i(\Alb X, {a_X}_*\omega_X)$ in the maximal Albanese dimension  case. 
 
  It has been conjectured by Popa (\cite{popa-conj1}) that all loci $V^i(X,\omega_X)$ are derived invariants of smooth complex projective varieties. This conjecture has been verified by Lombardi and Popa -- only for the components containing the origin of $\Pic0 X$ -- unconditionally on the Albanese dimension 
      for $i=0, 1,\dim X-1,\dim X$ (\cite{lombardi-derived},\cite{popa-conj2}) and
    in dimension  3  (\cite{lombardi-derived}), and for varieties of maximal Albanese dimension in dimension  4 (\cite{popa-conj2}). The following corollary fully proves   Popa's conjecture for varieties of maximal Albanese dimension, and, in general, the analogous statement for the loci $V^i_r(\Alb X, {a_X}_*\omega_X)$. 
  
  \begin{corollaryalpha} \label{cor2} Let $X$ and $Y$ be  varieties with equivalent derived categories. 
  For all $i,r\in\mathbb N$, the Rouquier isomorphism induces an isomorphism between 
 $V^i_r(\Alb X,{a_X}_*\omega_X)$ and  $V^i_r(\Alb Y,{a_Y}_*\omega_Y)$.

\noindent   In particular, if $X$ is of maximal Albanese dimension then all cohomological support loci
  $V^i_r(X,\omega_X)$ and  $V^i_r(Y,\omega_Y)$ are isomorphic.
   \end{corollaryalpha}  
   
   The method of proof of Theorem \ref{main} makes use of many essential results concerning the geometry of irregular varieties based on generic vanishing theory:  generic vanishing theorems, the relation between the loci $V^0(X,\omega_X^m)$ and the Iitaka fibration, the Chen-Jiang decomposition,   linearity theorems and their relation -- via the Bernstein-Gel'fand-Gel'fand correspondence -- with the Castelnuovo-Mumford regularity of suitable cohomology modules. This material is briefly reviewed in \S2 and \S3. The starting point of the argument are the  results and constructions of Lombardi (\cite{lombardi-derived}), who proves, in particular, the case $i=0$  of Theorem \ref{main}. Roughly, our method derives Theorem \ref{main} from the case $i=0$ by means of the derived invariance of the Hochschild multiplicative structure,  combined with the result of Lazarsfeld, Popa and Schnell on  the cohomology modules $H^*(\Alb X,{a_X}_*\omega_X\otimes P_\alpha)$ over the exterior algebra $\Lambda^*H^1(\Alb X, \OO_{\Alb X})$.
   
   Next, we turn to some applications of Theorem \ref{main} and especially of Corollary \ref{cor2}. It is known by the seminal work of Green and Lazarsfeld \cite{gl2} that the positive-dimensional components of the loci $V_r^i(X, \omega_X)$ are related to the presence of \emph{irregular fibrations} i.e. morphisms with connected fibres onto lower-dimensional  normal projective varieties -- here called \emph{base} of the fibration -- whose smooth models have maximal Albanese dimension. Therefore, as sought by Popa (\cite{popa-conj1}) and in the spirit of previous work of Lombardi and Popa (\cite{lombardi-derived}, \cite{popa-conj2}, and especially \cite{lombardi-fibrations}), the part of Corollary \ref{cor2} concerning varieties of maximal Albanese dimension implies the derived invariance of the presence/absence of certain irregular fibrations and, more, the invariance of the set itself of such fibrations. This imposes striking restrictions to the geometry and topology of the Fourier-Mukai partners. An example of this is the following Theorem \ref{fibrations},  concerning irregular fibrations of minimal base-dimension on varieties of maximal Albanese dimension. We remark that it is likely that a more thorough analysis of the informations provided by Theorem \ref{main} can get to more complete results.

   Turning to details, let us recall some notions appearing in the statement of Theorem \ref{fibrations}. In the first place we recall that  $\chi(\omega_{Z})\ge 0$ for a variety $Z$ of maximal Albanese dimension (see \S2). An irregular fibration
    \[ g:X\rightarrow S\]
     is said to be $\chi$-positive if $\chi(\omega_{S^\prime})>0$ for a smooth model $S^\prime$ of $S$ (hence for all of them). 
     This implies, in particular, that $S^\prime$ is of general type. $\chi$-positive fibrations might be seen as the higher-dimensional analogue of fibrations onto curves of genus $\ge 2$, which were classically studied by Castelnuovo and de Franchis (\cite{castel},\cite{defranchis}). Unconditionally on the Albanese dimension, Lombardi proved the invariance of the equivalence classes of the set of fibrations over curves of genus $\ge 2$ (\cite{lombardi-fibrations}).  For varieties of maximal Albanese dimension we note that, as a consequence of Orlov's theorem on the derived invariance of the canonical ring, the equivalence classes of all $\chi$-positive irregular fibrations are derived invariant (Prop \ref{chi-positive} below).
     
      On the other hand, even in the case of varieties of maximal Albanese dimension,  it is unclear what happens for non $\chi$-positive fibrations, especially when the base is birational to an abelian variety.
     Theorem \ref{fibrations} below gives a positive result about the derived invariance of the equivalence classes of a certain type irregular fibrations which are not necessarily $\chi$-positive, and include certain fibrations onto abelian varieties.   In order to state the result, let us denote  $\Pic0 (g)$ the kernel of the restriction of $\Pic0 X$ to a general fibre of $g$. It is well known that $\Pic0 (g)$ is an extension of $g^*\Pic0 S$ by a finite subgroup of $\Pic0 X / g^*\Pic0 S$, hence it may be disconnected.\footnote{For fibrations $g$ onto curves  the subvariety $\Pic0 (g)$ is completely described in the work of Beauville  \cite{beauville}.}
     We call  \emph{cohomologically detectable}  all irregular fibrations $g:X\rightarrow S$ 
    except those such that $S$ is birational to an abelian variety $B$ and $\Pic0 (g)=g^*\Pic0 B$, and we denote by $b(X)$ the minimal base-dimension (namely $\dim S$) of such fibrations (if there are no cohomologically detectable fibrations we declare that $b(X)=0$).  The explanation for such terminology is in Remark \ref{bad} below. Suffice here to say that, by the Green-Lazarsfeld linearity theorem, "most" irregular fibrations are induced by positive-dimensional components of the various loci $V^i(X,\omega_X)$. For irregular fibrations $g$ whose base  is birational to an abelian variety, and $\Pic0 (g)$ is connected, either this doesn't happen or it happens in a non-standard way.

      \begin{theoremalpha}\label{fibrations} Let $X$ and $Y$ be $d$-dimensional derived equivalent varieties of maximal Albanese dimension. Then:
      \[b(X)=b(Y):=b. \]
    Moreover there is a base-preserving  bijection
          of the sets  of the equivalence classes  of cohomologically detectable  irregular fibrations  of $X$ and $Y$ of base-dimension equal to $b$.
           Such bijection takes $\chi$-positive fibrations to  $\chi$-positive fibrations.
           \end{theoremalpha}

       Theorem \ref{fibrations} is proved  by considering a special class of components  of the loci $V^i(X,\omega_X)$ --   studied by the second author in  \cite{standard} -- whose relation with the corresponding fibration (via the theorem of Green and Lazarsfeld) is somewhat standard.  We first show a fact of independent interest (Lemma \ref{bijection}), namely  that these 
       ``standard'' components are translates of abelian subvarieties for which there is a natural bijective correspondence with the set of equivalence classes of another type of  fibrations, called \emph{weakly-$\chi$-positive}, containing the $\chi$-positive ones. However, in general there is no easy way to distinguish the standard components from the non-standard ones. In the second part of the proof of Theorem \ref{fibrations} we show that this can be done in the locus $V^{d-b(X)}(X,\omega_X)$ and in this case the weakly $\chi$-positive fibrations coincide with the cohomologically detectable ones. In this way  Theorem \ref{fibrations} follows from Corollary \ref{cor2}.

   \noindent 
   Finally we remark  that  Theorem \ref{main} provides also some information  about derived invariance of fibrations of varieties of arbitrary Albanese dimension. In fact a well known argument using Koll\'ar decomposition shows that  positive-dimensional  irreducible components of the loci $V^i_r(\Alb X,{a_X}_*\omega_X)$ form a subset of the set of the irreducible components of the loci $V^i_{r^\prime}(X,\omega_X)$ for some $r^\prime\ge r$. Hence, via the Green-Lazarsfeld theorem, they correspond to some irregular fibrations.  However at present it is not clear to us how to describe them.
   
 We will work over $ \mathbb C$. All varieties appearing in this paper are assumed to  be projective. A variety without further specification means smooth complex projective variety. Normal variety means normal projective variety. An Albanese morphism  means an universal morphism from a fixed variety $X$ to abelian varieties. We will call such a morphism the Albanese morphism or also the Albanese map of $X$, and we will denote it $a_{X}:X\rightarrow \Alb X$.
 
We thank Luigi Lombardi for answering our questions and the referee for his careful reading and his helpful comments to improve the exposition of this text.

\section{Preliminary material on generic vanishing, Chen-Jiang decomposition and 0-regularity of the canonical module } 
In this section we recall material used in the sequel, referring to the appropriate sections of papers as \cite{msri}, \cite{hps}, \cite{ps2}, \cite{pps}, \cite{standard} for more thorough surveys. 
 For a morphism of abelian varieties $\pi: A\rightarrow B$ we will denote 
\[\widehat \pi:\Pic0 B\rightarrow \Pic0 A\]
 the dual morphism.

\noindent\textbf{Generic vanishing. }
 Let $A$ be an abelian variety.  
A coherent sheaf $\G$ is said to be a \emph{generic vanishing sheaf}, or \emph{GV-sheaf} for short, if 
\[  \mathrm{codim}_{\Pic0 A}V^i(A,\G)\ge i\quad\hbox{for all $i\ge 0$}
\]
$\G$ is said to be \emph{M-regular} if 
\[\mathrm{codim}_{\Pic0 A}V^i(A,\G)>i\quad\hbox{for all $i>0$}\]

\begin{remark}\label{easy} If $\G$ is GV then $\chi(\G)\ge 0$ and $\chi(\G)>0$ if and only if $V^0(A,\G)=\Pic0 A$. 
\end{remark}

We have the following well known non-vanishing results (see e.g. \cite{hacon} Cor. 3.2 for (a) and \cite{msri} Lemma 1.12 for (b) and (c)) 
\begin{proposition}\label{non-vanish}Let $\G$ be a non-zero coherent sheaf on an abelian variety $A$, \\
(a) If $\G$ is GV then $V^{i+1}(A,\G)\subseteq V^i(A,\G)$ for all $i\ge 0$.\\
(b) If $\G$ is  GV then $V^0(A,\G)\ne \emptyset$.\\
(c) If $\G$ is M-regular then $V^0(A,\G)=\Pic0 A$.
\end{proposition}

 \vskip0.3truecm\noindent\textbf{Chen-Jiang decomposition. } 
  This concept was introduced by 
 J.A. Chen and Z. Jiang (\cite{chen-jiang} Theorem 1.1). The following theorem was proved in \cite{pps}. Actually in this paper we will use only the case $j=0$.

\begin{theorem}\label{CJ}(Chen-Jiang decomposition). Let $a:X\rightarrow A$ be a morphism from a
 variety to an abelian variety and let $j\ge 0$. Then the sheaf $R^ja_*\omega_X$  decomposes canonically as
$$R^j{a}_*\omega_X =\bigoplus_{i}\pi_i^*\F_i\otimes P_{\alpha_i}$$
where $\pi_i: A\rightarrow B_i$ are quotients of abelian varieties with connected fibres , $\F_i$ are $M$-regular sheaves on $B_i$ and $\alpha_i$ are torsion points of $\Pic0 A$. 
\end{theorem}
Note that in the above decomposition we can  arrange that
$\widehat\pi_i(\Pic0 B_i)-\alpha_i\ne \widehat\pi_k(\Pic0 B_k)-\alpha_k$ for $i\ne k$. With this normalization the decomposition is canonical up to permutation of the summands. 

\begin{remark}\label{after-CJ} Theorem \ref{CJ} has the following consequences:

\noindent  (1) \emph{For all $j\ge 0$ the sheaf $R^j{a}_*\omega_X$ is a GV-sheaf on $A$ } (Hacon \cite{hacon}).\\
This is because, by projection formula, the pullback of a GV-sheaf via a morphism of abelian varieties is still GV.

\noindent (2) \ \ $V^0(A,R^ja_*\omega_X)=\bigcup_i (\widehat{\pi_i}(\Pic0 B_i)-\alpha_i)$.

\noindent  This last equality   again follows from projection formula: 
\begin{equation}H^0(A,\pi_i^*\F_i\otimes P_{\alpha_i}\otimes P_\alpha)=\begin{cases}H^0(B_i,\F_i\otimes P_\beta)\>\>\>\hbox{if $\alpha=\widehat{\pi_i}(\beta)-\alpha_i$ with $\beta\in \widehat{\pi_i}(\Pic0 B_i)$}\\0\>\>\>\hbox{otherwise}\end{cases}
\end{equation}
This, together with Proposition \ref{non-vanish}(c), shows that the locus $V^0(A, R^ja_*\omega_X)$ is the union of translates of the abelian subvarieties $\widehat{\pi_i}(\Pic0 B_i)$ of $\Pic0 A$ by points of finite order.\footnote{By a theorem of Green-Lazarsfeld and Simpson actually this is true -- and of fundamental importance -- for all loci $V^i_r(A, R^ja_*\omega_X)$ for all $i,j$ and $r$, see \S4.}

\noindent (3) Keeping the notation of Theorem \ref{CJ}, let $c(i)=\dim A-\dim B_i$. Again  from projection formula,  combined with Prop. \ref{non-vanish}(c), it follows  that the support of $V^{c(i)}(A, \pi_i^*\F_i\otimes P_{\alpha_i})$ is equal to the support of $V^{0}(A, \pi_i^*\F_i\otimes P_{\alpha_i})$, namely $\widehat{\pi_i}(\Pic0 B_i)-\alpha_i$. This implies a result originally due to Ein-Lazarsfeld (\cite{el}): the irreducible components of the locus $V^0(A, R^j{a}_*\omega_X)$ of codimension $c>0$ are also components
of the locus $V^c(A,R^j{a}_*\omega_X)$. 
\end{remark}

\begin{remark}\label{twist}
 Theorem \ref{CJ} and its consequences hold more generally for the sheaves $R^ja_*(\omega_X\otimes  P_\alpha)$ 
where $\alpha$ is a torsion point of $\Pic0 X$. 
This is because $\omega_X\otimes P_\alpha$ is a direct summand of $f_*\omega_{\widetilde X}$, for   a suitable  \'etale cover  $f:\widetilde X\rightarrow X$. 
\end{remark}

\noindent\textbf{Strong linearity and Castelnuovo-Mumford regularity } The relation between the theory of generic vanishing 
 and the Bernstein-Gel'fand-Gel'fand correspondence was pointed out in the paper \cite{bgg}, and further developed in
 \cite{lps} and \cite{ps2}. 
 
 For a sheaf $\G$ on an abelian variety $A$ let us consider its cohomology module 
 \begin{equation}\label{notation}H^*(A,\G)=\bigoplus_iH^i(A,\G)
 \end{equation}
  which is a graded module over the exterior algebra $\Lambda^*H^1(A,\OO_A)$. For such a graded module there is the notion of Castelnuovo-Mumford regularity. In particular, $\reg(H^*(A,\G))=0$ if and only if it is generated in degree $0$ and it has a linear graded free resolution. In the sequel we will use  the case $j=0$ of the following Theorem of Lazarsfeld, Popa and Schnell (\cite{lps} Thm 2.1)

\begin{theorem}\label{0reg} (Lazarsfeld-Popa-Schnell).  Let $a:X\rightarrow A$ be a morphism from a variety $X$ to an abelian variety $A$.  Let  $\alpha\in \Pic0 X$ be a torsion point, and let $\beta\in \Pic0 A$. Then, for all $j\ge 0$
\[\reg\bigl(H^*(A,R^ja_*(\omega_X\otimes P_\alpha)\otimes P_\beta)\bigr)=0\]
\end{theorem}

\noindent Note that this theorem is stated in \cite{lps} in a more restrictive setting, namely only for $R^j{a_X}_*\omega_X$, where $a_X$ denotes the Albanese morphism of $X$.
However the proof of the result goes through without any change. The point here is that the Green-Lazarsfeld's theorem about computing the higher direct images of the Poincar\'e bundle by means of the derivative complex (\cite{gl2} \S3) holds in the neighborhood of every point in $\Pic0 A$, so that the machinery in \cite{bgg} and \cite{lps} applies.

 \section{Proof of Theorem \ref{main}}

\vskip0.3truecm \noindent\textbf{Preliminaries: two results of Lombardi. } Again, for sake of brevity we will state only those results strictly needed for our arguments,  referring to the paper \cite{lombardi-derived} for the complete story.
Let \break
$\varphi_{}:\mathbf D(X)\rightarrow \mathbf D(Y)$ be an exact equivalence and
\[\overline\varphi:\mathrm{Aut}^0\, X \times \Pic0 X \rightarrow \mathrm{Aut}^0\, Y \times \Pic0 Y\]
its Rouquier isomorphism. 
The following result  (\cite{lombardi-derived} Prop. 3.1) will be fundamental for our arguments

\begin{theorem}\label{lomb1} Let $m$ be an integer, and assume that $h^0(X, \omega_X^m\otimes P_\alpha)>0$. Then $\bar\varphi(\(id)_X,P_\alpha)$ is of the form $(\(id)_Y,P_\beta)$ for  $\beta\in\Pic0 Y$.  
\emph{If this is the case we will abusively denote
\begin{equation}\label{abuse}\beta=\bar\varphi(\alpha)
\end{equation}}
\end{theorem}

Let us denote $\delta:X\rightarrow X\times X$ the diagonal morphism. Again following Lombardi (\cite{lombardi-derived}) for fixed 
$m\in \mathbb Z$ and $\alpha\in\Pic0 X$ we consider the \emph {twisted  (generalized) Hochschild homology}\footnote{as already mentioned, the setting of Lombardi is more general. Here we are stating only what is necessary for our purposes}
\begin{equation}\label{coho}HH_*^m(X,\alpha)=\bigoplus_k\mathrm Ext^k_{\OO_{X\times X}}(\delta_*\OO_X,\delta_*(\omega^m_X\otimes P_\alpha)).
\end{equation}
It is a graded module over the \emph{Hochschild cohomology algebra}
$$HH^*(X)=\bigoplus_{k}\mathrm{Ext}^k_{\OO_{X\times X}}(\delta_*\OO_X,\delta_*\OO_X).$$ 
A classical result  of Orlov and C\u{a}ld\u{a}raru (\cite{orlov}, \cite{calda}), generalized by Lombardi  to the twisted case (\cite{lombardi-derived} Thm 1.1)  proves 

\begin{theorem}\label{lomb2} In the above notation let $m \in \mathbb Z$ and $\alpha\in\Pic0 X$ such  that $h^0(X,\omega_X^m\otimes P_\alpha)>0$. Then  the derived equivalence
$\varphi$   induces  a canonical graded-algebra isomorphism
$$\Phi^*:HH^*(X)\rightarrow HH^*(Y)$$
and, using notation (\ref{abuse}), a compatible graded-module isomorphism 
\begin{equation}\label{coho2}\Phi_{\alpha,*}^m:HH_*^m(X,\alpha)\rightarrow HH_*^m(Y,\bar\varphi(\alpha)).
\end{equation} 
In particular, in degree $0$, \  $HH^m_0(X,\alpha)=H^0(X,\omega^m_X\otimes P_\alpha)$, \  hence we have the isomorphism
\begin{equation}\label{lomb-formula}\Phi_{\alpha,0}^{m}:H^0(X,\omega_X^m\otimes P_\alpha)\buildrel\sim\over\rightarrow H^0(Y,\omega_Y^m\otimes  P_{\overline\varphi(\alpha)})
\end{equation} 
\end{theorem}

Going back to the Rouquier isomorphism, it follows that, for all $m\in \mathbb Z$ and $r\ge 1$
\begin{equation}\label{lomb3}\overline\varphi\bigl( \{\(id)_X\} \times V^0_{r}(X,\omega_X^m)\bigr)
= \{\(id)_Y\} \times V^0_{r}(Y,\omega_Y^m)\end{equation}

 For $m=1$ we will suppress, as it is customary, the index $1$ in (\ref{coho}) and in (\ref{coho2}).

\vskip0.3truecm \noindent\textbf{Preliminaries: the Iitaka fibration of irregular varieties. } 
Assume that the Kodaira dimension of $X$ is non-negative. Then the loci $V^0(X,\omega_X^m)$ are tightly connected with the Iitaka fibration of $X$.

After a birational modification of $X$ we can assume that the Iitaka fibration of $X$ is a morphism $X\rightarrow Z_X$ with $Z_X$ smooth.  There is the commutative diagram
$$\xymatrix{X\ar[r]^{a_X}\ar[d]^{f_X}&\mathrm{Alb}\, X\ar[d]^{a_{f_X}}\\Z_X\ar[r]^{a_{Z_X}}&\mathrm{Alb}\, Z_X}$$
where $a_{f_X}$ is a surjective morphism of abelian varieties with connected fibres (\cite{hps} Lemma 1.11(a)).
We will make use of the following results of Chen-Hacon and Hacon-Popa-Schnell:

\begin{theorem}\label{chps}\emph{(a) (\cite{hps} Th. 11.2(b)) } For $m\ge 2$, the irreducible components of the locus $V^0(X,\omega_X^m)$ are translates of \ $\widehat{a_{f_X}}(\Pic0 Z_X)$  by torsion points of $\Pic0 X$.

\noindent \emph{(b)  (\cite{ch-iitaka1} Lemma 2.2, see also  \cite{hps}  (2) after Lemma 11.1)   } The irreducible components of the locus $V^0(X,\omega_X)$ are translates  
 by torsion points of $\Pic0 X$ of 
abelian subvarieties of the abelian subvariety \ $\widehat{a_{f_X}}(\Pic0 Z_X)$. 
\end{theorem}

\vskip0.3truecm \noindent\textbf{Proof of Theorem \ref{main}. } Let $\alpha\in\Pic0 X$ and $i\ge 0$ such that 
\begin{equation}\label{hyp}h^i(\Alb X, \, {a_X}_*\omega_X\otimes P_\alpha)>0
\end{equation}

\noindent\textbf{Step 1.} \emph{The Kodaira dimension of $X$ and $Y$ are non-negative. }

\proof 
Indeed by (\ref{hyp})  \ $V^i(\Alb X, \,{a_X}_*\omega_X)\ne \emptyset$. Therefore $V^0(\Alb X,{ \,a_X}_*\omega_X)\ne \emptyset$ by Proposition \ref{non-vanish}(a). By Remark \ref{after-CJ}(2)  this yields that $V^0(\Alb X,\,{a_X}_*\omega_X) = V^0(X,\omega_X)$ contains
some  points $\alpha$ of $\Pic0 X$ of finite order, say $k$. This implies that $h^0(X,(\omega_X\otimes P_\alpha)^k)=h^0(X,\omega_X^k)>0$. Therefore $\kappa(X)\ge 0$.   Since the Kodaira dimension is a derived invariant, the same holds for $Y$.\endproof

We have natural embeddings 
\[H^1(Z_X, \OO_{Z_X})\subset H^1(X, \OO_X)\subset HH^1(X)=\mathrm{Ext}^1_{\OO_{X\times X}}(\delta_*\OO_X,\delta_*\OO_X).\]
The same for $Y$. 

\vskip0.3truecm\noindent\textbf{Step 2.} The second step is the following 

\begin{lemma}\label{referee}\ \  $\Phi^1H^1(Z_X, \OO_{Z_X})=H^1(Z_Y, \OO_{Z_Y})$. 
\end{lemma}
\proof This follows at once from the above results. 
Indeed, combining Theorem \ref{chps}(a) and (\ref{lomb3}) we get that
\begin{equation}\label{picz}\overline\varphi(\widehat{a_{f_X}}(\Pic0 Z_X))=\widehat{a_{f_Y}}(\Pic0 Z_Y)
\end{equation}
On the other hand,  it is well known that the isomorphism $\Phi^1$ i.e.\footnote{the spectral sequence abutting to $\mathrm{Ext}^i_{X\times X}(\delta_*\OO_X,\delta_*\OO_X)$ degenerates, see \cite{swan} Cor. 2.6.}
$$\xymatrix{\mathrm{Ext}^1_{X\times X}(\delta_*\OO_X,\delta_*\OO_X)\ar[r]^{\Phi^1}\ar[d]^\sim&\mathrm{Ext}^1_{Y\times Y}(\delta_*\OO_Y,\delta_*\OO_Y)\ar[d]^\sim\\
H^0(T_X) \oplus H^1(\OO_X)\ar[r]&H^0(T_Y) \oplus H^1(\OO_Y)}$$
is the first order version of the Rouquier isomorphism (see e.g. \cite{huybrechts} p.218). Therefore
Step 2 follows from (\ref{picz}).\endproof

Next, we note that, by Theorem \ref{chps}(b), we can gather those irreducible components of $V^0(X,\omega_X)=V^0(\Alb X,{a_X}_*\omega_X)$ which are
contained in the same translate of $\widehat{a_{f_X}}(\Pic0 Z_X)$. Hence, using Remark  \ref{after-CJ}(2), we can gather the corresponding sheaves appearing in the Chen-Jiang decomposition of ${a_X}_*\omega_X$ 
yielding another  canonical decomposition 
\begin{equation}\label{decomp}{a_X}_*\omega_X=\bigoplus_{j=1}^{r_X}(a_{f_X}^*\H_{X,j})\otimes P_{X, \delta_j}
\end{equation}
defined by the following properties: \\
\emph{ the $\H_{X,j}$'s are GV-sheaves  on $\Alb Z_X$  \emph{(in fact the direct sum of some pullbacks of $M$-regular sheaves from quotient abelian varieties appearing in the Chen-Jiang decomposition of ${a_{X}}_*\omega_X$)}, the $\delta_j$ are torsion points of $\Pic0 X$ and $r_X$ is the number of translates  in $\Pic0 X$ of the abelian subvariety $\widehat
{a_{f_X}}(\Pic0 Z_X)$    containing at least one component of the locus $V^0(\Alb X,{a_X}_*\omega_X)$. }\\
The same sort of decomposition holds for ${a_Y}_*\omega_Y$:
$${a_Y}_*\omega_Y=\bigoplus_{k=1}^{r_Y}(a_{f_Y}^*\H_{Y,k})\otimes P_{Y, \gamma_k}.$$
We claim that 
 \[r_X=r_Y:=r\]
 and, up to reordering, for all $j=1,\dots ,r$
\[\overline\varphi\Bigl(V^0(X,\omega_X)\cap(\widehat{a_{f_X}}(\Pic0 Z_X)-\delta_j)\Bigr)=V^0(Y,\omega_Y)\cap (\widehat{a_{f_Y}}(\Pic0 Z_Y)-\gamma_j).\]
In fact each component of $V^0(X,\omega_X)$ (which is a translate of an abelian subvariety $\widehat{a_{f_X}}(\Pic0 Z_X)$) is contained in a unique translate of $\widehat{a_{f_X}}(\Pic0 Z_X)$. The same happens on $Y$. From (\ref{picz}) and Lombardi's theorem  (\ref{lomb3})  it follows that the algebraic group isomorphism $\overline\varphi$ sends such a translate of $\widehat{a_{f_X}}(\Pic0 Z_X)$ to the corresponding translate (in $\Pic0 Y$) of $\widehat{a_{f_Y}}(\Pic0 Z_Y)$. This proves what claimed.

 In fact, since two different translates have empty intersection, we have that:  \\
(*) \emph{  for $i\ge 0$ and for a fixed $\alpha\in\Pic0X$,
 in the decomposition
$$H^i({a_X}_*\omega_X\otimes P_\alpha) =\bigoplus_{j=1}^{r_X}H^i((a_{f_X}^*\H_{X,j})\otimes P_{\delta_j+\alpha})$$
at most  one summand is non-zero. } \\
For $i=0$ this holds by definition of the above decomposition, and for $i>0$ it follows as above from Proposition \ref{non-vanish} (a). 
Moreover, from projection formula and the fact that the quotient $\Alb X\rightarrow \Alb Z_X$ has connected fibres it follows that 
 \begin{equation}\label{projection-f-0}H^0(\Alb X, (a_{f_X}^*\H_{X,j})\otimes P_{\delta_j+\alpha})\\
 \\=\begin{cases}H^0(\Alb Z_X,\H_{X,j}\otimes P_\eta)\>\>\>\hbox{if $\delta_j + \alpha=\widehat{a_{f_X}}(\eta)$ with $\eta\in\Pic0 Z_X$}\\0\>\>\>\hbox{otherwise}\end{cases}
 \end{equation}
 The same holds for $Y$. This, combined with  (\ref{lomb-formula}) proves:
 
 \vskip0.3truecm\noindent\textbf{Step 3.} \emph{Keeping the above notation  let $\alpha\in V^0(\Alb X,{a_X}_*\omega_X)$ and $\eta\in \Pic0 Z_X$  such that $\widehat{a_{f_X}}(\eta)=\alpha+\delta_j$. Then
 \[\Phi_{\alpha+\delta_j}^0\, H^0(\Alb Z_X,\H_{X,j}\otimes P_\eta)= H^0(\Alb Z_Y,\H_{Y,j}\otimes P_{\overline\varphi(\eta)})\]
 where, via a slight abuse of language, we are denoting $\overline\varphi(\eta)\in\Pic0 Z_Y$ the element $\nu\in \Pic0 Z_Y$ such that, by (\ref{picz}), $\widehat{a_{f_Y}}(\nu)=\overline \varphi(\widehat{a_{f_X}}(\eta))$.}

\vskip0.3truecm Next, we recall that for all $\alpha\in\Pic0 X$ the local to global spectral sequence computing each graded component $HH_i(X,\alpha)$ degenerates (\cite{swan} Cor. 2.6). It follows that the canonical map from  $H^i(X,\omega_X\otimes P_\alpha)$ to
 $HH_i(X,\alpha)$ is an embedding. 
More, for $\alpha\in V^0(\Alb X,\,{a_X}_*\omega_X)$ and $\eta\in \Pic0 Z_X$  such that $\widehat{a_{f_X}}(\eta)=\alpha+\delta_j$, we have the following chain of canonical  embeddings of vector spaces
\begin{equation}\label{inclusion1}
H^i(\Alb Z_X,\H_{X,j}\otimes P_\eta)\hookrightarrow H^i(\Alb X, \,{a_X}_*\omega_X\otimes P_\alpha)\hookrightarrow H^i(X,\omega_X\otimes P_\alpha)\hookrightarrow HH_i(X,\alpha)
\end{equation}
(and the same things holds for $Y$).\footnote{In the second space $P_\alpha$ denotes a line bundle on $\Alb X$, while in the third space $P_\alpha$ denotes a line bundle on $X$, i.e., strictly speaking, the pullback, via the Albanese map, of the previous $P_\alpha$.} The first inclusion follows from (\ref{decomp}) via projection formula, and the second one follows 
from Koll\'ar's theorem on the degeneration of the Leray spectral sequence of the canonical bundle (\cite{kollar2}), once again combined with projection formula.

\noindent\textbf{Step 4.}  \emph{ Let $\alpha\in V^0(\Alb X,{a_X}_*\omega_X)$ and $\eta\in \Pic0 Z_X$  such that $\widehat{a_{f_X}}(\eta)=\alpha+\delta_j$. Then, for all $i\ge 0$},
 \[\Phi_{\alpha+\delta_j}^i\, H^i(\Alb Z_X,\H_{X,j}\otimes P_\eta)= H^i(\Alb Z_Y,\H_{Y,j}\otimes P_{\overline\varphi(\eta)}).\]
 
\proof It is here where we use 
 the multiplicative structure and the BGG correspondence. We have the following morphisms of graded algebras
\[\Lambda^*H^1(Z_X, \OO_{Z_X})\hookrightarrow  \Lambda^*H^1(X, \OO_X)\rightarrow H^*(X, \OO_X)\hookrightarrow HH^*(X)\]
Therefore the Hochschild cohomology $HH_*(X,\alpha)$ is  a graded module also on all the graded algebras appearing above. The similar thing holds for $HH_*(Y,\beta)$. 
From Step 2 it follows that:

\noindent (**) \emph{ the  graded module isomorphism $\Phi_\alpha: HH_*(X,\alpha)\buildrel\sim\over\rightarrow HH_*(Y,\overline\varphi(\alpha))$ of Theorem \ref{lomb2} is compatible with the isomorphism  $\wedge^*(\Phi^1): \Lambda^*H^1(\OO_{Z_X})\buildrel\sim\over\rightarrow \Lambda^*H^1(\OO_{Z_Y})$.}

\noindent The inclusions (\ref{inclusion1}) fit into inclusions of graded modules over the exterior algebra 
$\Lambda^*H^1(\OO_{Z_X})$
\[H^*(\Alb {Z_X}, \H_{X,j}\otimes P_\eta)\hookrightarrow H^*(\Alb X, \,{a_X}_*\omega_X\otimes P_\alpha)\hookrightarrow H^*(X,\omega_X\otimes P_\alpha)\hookrightarrow HH_*(X,\alpha).\]
 For $\alpha\in V^0(\Alb X,{a_X}_*\omega_X)$ and $\eta\in \Pic0 Z_X$  such that $\widehat{a_{f_X}}(\eta)=\alpha+\delta_j$ let us denote 
  \[\widetilde H^*(\Alb Z_X,\H_{X,j}\otimes P_\eta)\]
    the graded  $\Lambda^*H^1(\OO_{Z_X})$-submodule of $HH_*(X,\alpha)$ generated by $H^0(\Alb {Z_X}, \H_{X,j}\otimes P_\eta)$. Clearly 
  \begin{equation}\label{modules}\widetilde H^*(\Alb Z_X,\H_{X,j}\otimes P_\eta)\subseteq H^*(\Alb Z_X, \H_{X,j}\otimes P_\eta). 
  \end{equation}
  (in fact the first is a submodule of the second).
    By Step 3 and (**) it follows that
  \begin{equation}\label{gen}\Phi_{\delta_j+\alpha} \widetilde H^*(\Alb Z_X,\H_{X,j}\otimes P_\eta)=\widetilde H^*(\Alb Z_Y,\H_{Y,j}\otimes P_{\overline\varphi(\eta)}).
  \end{equation}
  By projection formula on the decomposition (\ref{decomp}) it follows that the sheaf $\H_{X,j}\otimes P_\eta$ is a direct summand of the sheaf ${a_{f_X}}_*({a_X}_*\omega_X\otimes P_{\delta_j}^\vee)\otimes P_\eta$. Therefore the module 
  $H^*(\Alb Z_X, \H_{X,j}\otimes P_\eta)$ is a direct summand of the module $H^*(\Alb Z_X, {a_{f_X}}_*({a_X}_*\omega_X\otimes P_{\delta_j}^\vee)\otimes P_\eta)$, which is $0$-regular by  Theorem \ref{0reg}, hence, in particular, generated in degree $0$. Therefore the module  $H^*(\Alb Z_X, \H_{X,j}\otimes P_\eta)$ is generated in degree $0$ as well, and we have equality in (\ref{modules}).
 By the same reason the same thing happens for $Y$. Therefore Step 4 follows from (\ref{gen}). 
 \endproof

\vskip0.3truecm\noindent\textbf{Step 5.} \emph{Conclusion of the proof. } 
 Let $q=\dim \Alb X=\dim \Alb Y$ (Theorem of Popa-Schnell, \cite{ps1}), and let $q^\prime=\dim \Alb Z_X=\dim \Alb Z_Y$  (Step 2). Note that, since the quotient map $a_{f_X}$ has connected fibres, $R^k{a_{f_X}}_*\OO_{\Alb X}$ is a trivial bundle of rank ${{q-q^\prime}\choose{k}}$. Therefore for $\eta\in \Pic0 Z_X$  such that $\widehat{a_{f_X}}(\eta)=\alpha+\delta_j$
 we have   that
\begin{eqnarray*}h^i(\Alb X, {a_X}_*\omega_X\otimes P_{\alpha})&=&\bigoplus_{k=0}^{q-q^\prime} h^{i-k}(\Alb Z_X, R^k{a_{f_X}}_*({a_X}_*\omega_X)\otimes P_{\eta}))\\
&=&\bigoplus_{k=0}^{q-q^\prime} h^{i-k}(\Alb Z_X,\H_{X,j}\otimes P_\eta)^{\oplus {{q-q^\prime}\choose k}} 
\end{eqnarray*} 
where the first equality is the Koll\'ar decomposition (plus projection formula) with respect to the morphism $a_{f_X}$ applied to the sheaf ${a_X}_*\omega_X$ (\cite{kollar2} Th. 3.4) and the second equality follows from (*) and projection formula. 
The same formula holds for $Y$. Therefore Theorem \ref{main} follows from Step 4 applied to the last quantity. 

\section{Application to irregular fibrations: Theorem \ref{fibrations}}

\noindent\textbf{Fibrations: terminology. }
 Let $X$ be a variety. A \emph{fibration of X} is an algebraic fiber space $g: X\rightarrow S$ where  $S$ is a normal variety, called base of the fibration. If a non-singular model of $S$ (hence all of them) has maximal Albanese dimension such a fibration is said to be \emph{irregular}.\\
 A \emph{non-singular representative of a   fibration of $X$} is a fibration $g^\prime: X^\prime \rightarrow S^\prime$ with both $X^\prime$ and $S^\prime$ smooth, equipped with birational morphisms $p:X^\prime\rightarrow X$ and $q: S^\prime\rightarrow S$ such that $g\circ p=q\circ g^\prime$.\\
  Two fibrations of $X$ are  \emph{equivalent} if there is a fibration $X^\prime\rightarrow S^\prime$ which is a birational representative for both of them. \\
Let $g$ be a fibration of $X$.  We denote  $\Pic0 (g)$ the kernel of the restriction map from $\Pic0 X$ to $\Pic0$ of a general fibre. Notice that if $g^\prime$ is any non-singular representative of $g$ then $\Pic0(g)=\Pic0(g^\prime)$, therefore $\Pic0 (g)$ depends only on the equivalence class of $g$. $\Pic0 (g)$ is an extension of $g^*\Pic0 S$ by a finite subgroup $\Gamma$ of $\Pic0 X / g^*\Pic0 S$ (see e.g. \cite{standard}), hence it disconnected unless $\Gamma=\hat 0$.

\begin{definition} Let $g:X\rightarrow S$ be an irregular fibration of $X$ and let us denote $i=\dim X-\dim S$.

\noindent
 (a) $g$ is \emph{cohomologically non-detectable} if  $S$  birational to an abelian variety and $\Pic0 (g)$ is connected, and \emph{cohomologically detectable} otherwise. 
 
\noindent (b) $g$ is \emph{weakly-$\chi$-positive} if there is a point $\alpha\in\Pic0 X$ such that for a non-singular representative $g^\prime:X^\prime \rightarrow S^\prime$ (hence for all of them, see Remark \ref{previous} below)  
 \begin{equation}\label{chi}\chi(R^ig^\prime_*(\omega_{X^\prime}\otimes P_\alpha))>0.
 \end{equation}
  Note that an $\alpha\in\Pic0 X$ as in the definition must belong to $\Pic0 (g)$. Therefore 
 one can always assume that $\alpha$ is a torsion point.
 
\noindent  (c) $g$ is \emph{$\chi$-positive} if for  a non-singular representative $g^\prime$ as above (hence for all of them) $\chi(\omega_{S^\prime})>0$. 
\end{definition}
 Note that since $\omega_{S^\prime}=R^ig^\prime_*\omega_{X^\prime}$ (\cite{kollar1}, 
   Prop. 7.6) a $\chi$-positive irregular fibration is weakly-$\chi$-positive.

\begin{remark}\label{previous} We keep the notation of the above Definitions. From  Hacon's generic vanishing  (see Remark \ref{after-CJ}(1)) and an \'etale covering trick it follows  that $R^i(a_{S^\prime}\circ g^\prime)_*(\omega_{X^\prime}\otimes P_\alpha)$ is a GV-sheaf on $\Alb S^\prime$. On the other hand, since $a_{S^\prime}$ is generically finite, by the combination of Koll\'ar's vanishing and decomposition (\cite{kollar2} Thm 3.4), $R^k{a_{S^\prime}}_*R^h g^\prime_*(\omega_{X^\prime}\otimes P_\alpha)=0$ for all $k>0$ and $h\ge 0$, hence $R^i(a_{S^\prime}\circ g^\prime)_*(\omega_{X^\prime}\otimes P_\alpha)={a_{S^\prime}}_*R^ig^\prime_*(\omega_{X^\prime}\otimes P_\alpha)$. Therefore, having in mind Remark \ref{easy}, the condition  (\ref{chi}) is equivalent to the condition 
$$V^0(S^\prime, R^ig^\prime_*(\omega_{X^\prime}\otimes P_\alpha))=\Pic0 S^\prime.$$ 
This in turn implies that the condition (\ref{chi}) does not depend on the non-singular representative.
\end{remark}

\vskip0.3truecm\noindent\textbf{Preliminaries: the linearity theorem of Green and Lazarsfeld. } The relation between the loci $V^i(X, \omega_X)$ and irregular fibrations follows from the following fundamental theorem of Green and Lazarsfeld, with an addition of Simpson
\begin{theorem}[\cite{gl2},  \cite{simpson}]\label{gl2} Every irreducible component  $W$ of the loci   $V^i(X,\omega_X)$ is a \emph{linear subvariety} i.e. a translate of  an abelian subvariety $T\subset \Pic0 X$ by a torsion point. More precisely let
let $\pi: \mathrm{Alb}\, X\rightarrow B:=\Pic0 T$ the dual quotient. This defines the composed map $f:X\rightarrow B$ 
\begin{equation}\label{fundamental1}\xymatrix{X\ar[r]^{a_X}\ar[rd]^f& \mathrm{Alb}\, X\ar[d]^\pi\\
&B\\}
\end{equation} 
Then there is a torsion\footnote{this is due to Simpson} element  $\alpha\in\Pic0X$ such that
\begin{equation}\label{compo}W=\widehat\pi(\Pic0 B)+ \alpha . 
\end{equation} 
Moreover
\begin{equation}\label{fundamental2}
 \dim X-\dim f(X)\ge i.
\end{equation}
\end{theorem}

Taking the Stein factorization of the map $f$ one gets  a fibration $g:X\rightarrow S$, where $S$ is a normal projective variety of maximal Albanese dimension, and a finite morphism $a:S\rightarrow B$ such that $a\circ g=f$. Therefore, in our terminology, $g$ is an irregular fibration of $X$. We will refer to it  as \emph{the fibration of $X$ induced by the component $W$ of $V^i(X,\omega_X)$}, or also
 \emph{the fibration of $X$ induced by the abelian subvariety  $T$ of $\Pic0X$ parallel to the component $W$}.
In \cite{standard}, Lemma 5.1 it is shown in particular the following
\begin{proposition}\label{albanese} The above abelian variety $B$ is the Albanese variety of any non singular model of $S^\prime$ of $S$ and the morphism $a$, composed with the desingularization $S^\prime\rightarrow S$ is  an  Albanese morphism of $S^\prime$. In particular  \ $W=\widehat\pi(\Pic0 S^\prime)+\alpha$.
\end{proposition}
In conclusion, for a non-singular representative $g^\prime:X^\prime\rightarrow S^\prime$ of the induced fibration we have the commutative diagram
\begin{equation}\label{diagram}\xymatrix{X^\prime\ar@/^0.7pc/[rr]^{a_{X^\prime}}\ar[d]^{g^\prime}\ar[r]&X\ar[d]^g\ar[r]_{a_X}\ar[rd]^f&\Alb X\ar[d]^\pi\\
S^\prime\ar[r]\ar@/_0.7pc/[rr]_{a_{S^\prime}}&S\ar[r]^a&\Alb S^\prime}
\end{equation}

\noindent\textbf{Preliminaries: standard components and (weakly)-$\mathbf{\chi}$-positive irregular fibrations. }  We will suppose henceforth that \emph{$X$ has maximal Albanese dimension. }An irreducible component $W$ of $V^i(X,\omega_X)$ is said to be \emph{standard} (see \cite{standard}) if there is equality in (\ref{fundamental2}), i.e. 
\[\dim X-\dim S=i.\]
  The relation between standard components and their induced fibrations is almost canonical. This is the content of the following Lemma, inspired by Theorem 16 of \cite{lombardi-fibrations}. In the statement we consider the following sets:\\
  -  $\mathcal A(X)$ denotes \emph{the set of abelian subvarieties $T$ of the abelian variety $\Pic0 X$ such that some of their translates is a standard component of $V^i(X,\omega_X)$ for some index $i$} (clearly this can happen  for only one index $i$, denoted $i(T)$). \\
  - $\mathcal G(X)$ denotes \emph{the set of equivalence classes  of weakly-$\chi$-positive irregular fibrations of X. }\\
   \begin{lemma}\label{bijection} The function $\sigma:\mathcal A(X)\rightarrow \mathcal G(X)$ taking an abelian subvariety to the class of its induced fibration \emph{(see the above paragraph}) is a bijection. Moreover\\
  (1) $\sigma$ takes  those abelian subvarieties which are themselves (standard) 
components of $V^i(X,\omega_X)$ to the equivalence classes of $\chi$-positive fibrations .\\
(2)  the base-dimension of $\sigma(T)$ is $\le \dim T$.
\end{lemma}

\proof First we need to prove that if $T\in\mathcal A(X)$ then its induced fibration $g:X\rightarrow S$ is weakly-$\chi$-positive. Let $i=i(T)$, and let $W$ be a component verifying (\ref{compo}), with $T=\Pic0 B=\Pic0 S^\prime$ (see Prop. \ref{albanese} and (\ref{diagram}). By definition of standard component, $\dim X-\dim S=i$.
  Thanks to  Koll\'ar vanishing theorem (\cite{kollar1} Theorem 2.1) and decomposition (\cite{kollar2} Theorem 3.1), for a non-singular representative $g^\prime:X^\prime\rightarrow S^\prime$ of the fibration $g$ 
  one has that
\begin{equation}\label{union}V^i(X^\prime, \omega_{X^\prime}\otimes P_{-\alpha})=\bigcup_{j=0}^{i}\widehat\pi(V^{i-j}(S^\prime, R^jg^\prime_*(\omega_{X^\prime}\otimes P_{-\alpha})))
\end{equation}
where $\alpha\in\Pic0 X$ is the torsion point appearing in (\ref{compo}). Again by Hacon generic vanishing theorem (Remark \ref{after-CJ}(1)) and an \'etale covering trick, $\mathrm{codim}_{\Pic0 S^\prime} V^{i-j}(S^\prime, R^jg^\prime_*(\omega_{X^\prime}\otimes P_{-\alpha})\ge {i-j}$. 
Since, using also Prop. \ref{albanese}, the left hand side must contain $\widehat\pi(\Pic0 B)=\widehat\pi(\Pic0 S^\prime)$,  we have that
\begin{equation}\label{V0}V^0(S^\prime, R^{i}g^\prime_*(\omega_{X^\prime}\otimes P_{-\alpha}))=\Pic0 S^\prime
\end{equation}
 i.e., by Remark \ref{previous},
\[\chi(R^{i}g^\prime_*(\omega_{X^\prime}\otimes  P_{-\alpha}))>0.\]
This proves the desired assertion. 
By the same steps in the reverse order one proves that if $g:X\rightarrow S$ is  a weakly-$\chi$-positive irregular fibration  such that $\dim X-\dim S=i$ then (the equivalence class of) $g$  induces standard components $W$ in $V^i(X,\omega_X)$ as follows. Assume that $-\alpha\in\Pic0 (g)$ is such that $\chi(R^ig^\prime_*(\omega_{X^\prime}\otimes a_{X}^\prime P_{-\alpha}))>0$. Then
\[\widehat\pi(\Pic0 S^\prime)+\alpha=\widehat\pi(V^0(S^\prime, R^ig^\prime(\omega_{X^\prime}\otimes P_{-\alpha})))+\alpha\]  is a standard component of $V^i(X,\omega_X)$. 
 It is clear that the two constructions above are inverse to each other. Properties (1) and (2) are clear.
\endproof
\begin{remark}\label{bad}[Cohomologically non-detectable  fibrations.] The above argument with Koll\'ar decomposition proves also that a cohomologically non-detectable irregular fibration $g:X\rightarrow S$ can't be induced by a component $W$ of $V^i(X,\omega_X)$ of dimension $\ge \dim X-i$.
Indeed for such a fibration $V^0(S^\prime,\omega_{S^\prime})=\{\hat 0\}$ because $S$ is birational to an abelian variety. Therefore, since $\Pic0 (g)={g^\prime}^*\Pic0 S^\prime$, equality (\ref{V0}) can't hold. Since we know that (\ref{V0}) holds as soon as $\dim X-\dim S=i$ it follows that $\dim X-\dim S<i$, i.e. a component $W$ inducing such a fibration is non-standard. 
Moreover, since $\dim \Alb S^\prime=\dim S$, for such a component
\begin{equation}\label{added}\dim W<\dim X-i
\end{equation}
This explains the terminology \emph{cohomologically non-detectable irregular fibration}: such a fibration either it is not induced by any component of $V^i(X,\omega_X)$ for some $i$ (as for example  the  projections of a product of elliptic curves) or such a component is non-standard. 
\end{remark}

At the opposite end, $\chi$-positive fibrations are  the easiest to detect. The following proposition shows that equivalence classes of $\chi$-positive fibrations are derived invariants. 

\begin{proposition}\label{chi-positive} Let $X$ and $Y$ be varieties of maximal Albanese dimension with equivalent derived categories. Then there is a base-preserving bijection between the sets of  equivalence classes of $\chi$-positive irregular fibrations of $X$ and $Y$.
\end{proposition}
\begin{proof} By Lemma \ref{bijection} all $\chi$-positive fibrations on a variety $X$ of maximal Albanese dimension are induced by abelian subvarieties which are (standard) components of $V^i(X,\omega_X)$ for some $i$. By Prop. \ref{non-vanish}(a) such components are contained in $V^0(X,\omega_X)$, hence in $\widehat{a_{f_X}}(\Pic0 Z_X)$ (Theorem \ref{chps}(b)). Therefore $\chi$-positive fibrations, as all fibrations induced by components of $V^i(X,\omega_X)$ for some $i$, factor, up to equivalence, through the Iitaka fibration $X\buildrel{f_X}\over\rightarrow Z_X$.  But, by Orlov's theorem, a derived equivalence $\varphi:\mathbf D(X)\rightarrow \mathbf D(Y)$ induces an isomorphism of the canonical rings. Hence the bases of the Iitaka fibrations $Z_X$ and $Z_Y$ are birational. As we are considering equivalence classes of fibration, we can assume that $Z_X=Z_Y:=Z$. Therefore the sets of equivalence classes of $\chi$-positive irregular fibrations of $X$ and $Y$ are both naturally bijective with the set of equivalence classes $\chi$-positive fibrations of $Z$. 
\end{proof}

\vskip0.3truecm\noindent\textbf{Proof of Theorem \ref{fibrations}. } Let us recall that $b(X)$ denotes the minimal base-dimension of
 \emph{the cohomologically detectable} irregular fibrations of $X$. 

\noindent\textbf{Step 1.} Assume that $b(X)>0$. \\
\emph{\emph{(a)} An irregular fibration $g$ of  base-dimension equal to $b(X)$ is cohomologically detectable if and only if it is weakly-$\chi$-positive. Moreover it is  $\chi$-positive if and only if its base is not birational to an abelian variety. }\\
 \emph{\emph{(b)} Conversely, every irreducible component $W$ of $V^{d-b(X)}(X,\omega_X)$  such that $\dim W\ge b(X)$ is standard. If this is the case the abelian subvariety parallel to $W$ is also a component of $V^{d-b(X)}(X,\omega_X)$ if and only if the corresponding fibration (via Lemma \ref{bijection}) is $\chi$-positive.}

\noindent The argument for Step 1 is  well known to the experts (see e.g. \cite{msri}, proof of Lemma 4.2). We start with the following

\noindent\emph{Claim. Let $g:X\rightarrow S$  be a cohomologically detectable irregular fibration such that $\dim X-\dim S=i$. Then, keeping the notation above,  for at least one $\alpha\in\Pic0 (g)$, the locus 
\begin{equation}\label{locus}V^0(S^\prime, R^ig^\prime_*(\omega_{X^\prime}\otimes P_\alpha))
\end{equation}
 is positive-dimensional. }
 \proof We first observe that if $\alpha$ belongs to a component of $\Pic0 (g)$ different from the neutral one then 
  the locus (\ref{locus}) is positive dimensional. In fact it must be non-empty thanks to Prop. \ref{non-vanish}(b) and if it was $0$-dimensional this would induce via Remark \ref{after-CJ}(3)   a ($0$-dimensional) component of the locus $V^{q(S^\prime)}(R^ig^\prime_*(\omega_{X^\prime}\otimes P_\alpha)$. This implies that $\dim S^\prime=q(S^\prime)$ and, via the ever-present Koll\'ar decomposition as in (\ref{union}), this would induce some elements different from $\{\hat 0\}$ in the locus $V^d(X,\omega_X)$, which is impossible. 
 
 Therefore we are left with the case when $\Pic0 (g)$ is connected and $V^0(S^\prime,\omega_S^\prime)$ is zero-dimensional (recall that $R^ig^\prime_*\omega_{X^\prime}=\omega_{S^\prime})$. But this, by a Theorem of Ein-Lazarsfeld (\cite{ch1} Th.1.8) is equivalent to the fact that $S^\prime$ is birational to an abelian variety, i.e. the fibration would be non-detectable.\endproof
 
 We now turn to  Step 1(a). Let $g:X\rightarrow S$ be a cohomologically detectable fibration with $\dim S=b(X)$. We claim that if it is not weakly-$\chi$-positive then there is another
cohomologically detectable fibration of lower base-dimension factoring (up to equivalence) through $g$, in contradiction with the definition of $b(X)$. 
Let $\alpha\in \Pic0 (g)$ as in the Claim. Then, again by Remark \ref{after-CJ}(3), the irreducible components of codimension $c$, with $0 < c<q(S^\prime)$, of (\ref{locus}) are also irreducible components of $V^c(S^\prime, R^ig^\prime_*(\omega_{X^\prime}\otimes P_\alpha))$, where $i= d -\dim S$. Via  the Koll\'ar decomposition  they induce positive dimensional components of the locus
$V^{i+c}(\omega_X)$. Via the linearity theorem and Remark \ref{bad},  such a component induces another cohomologically detectable irregular fibration of $X$, say $h$, with $d-\dim h(X)\ge i + c = d-\dim S+c$. Hence $\dim h(X) \le \dim S -c$, as asserted. 
This proves the direct implication of  the first equivalence of (a). The other implication is clear. Passing to the second equivalence, the direct implication is clear. Conversely, let us suppose that the base is non-birational to an abelian variety. Then, by the Theorem of Ein-Lazarsfeld as above, $V^0(S^\prime, \omega_{S^\prime})$ is positive-dimensional. If it was strictly contained in $\Pic0 S^\prime$ then, as above, its components would induce a cohomologically detectable fibration $h$ of smaller base-dimension, against the definition of $b(X)$.  This completes the proof of (a).

\noindent Passing to Step1(b), let $W$ be a positive-dimensional component of $V^{d-b(X)}(X,\omega_X)$ such that $\dim W\ge b(X)$. The statement to prove is that the induced fibration $g$ has base-dimension equal to $b(X)$. If the base-dimension was $< b(X)$ then, by definition of the integer $b(X)$,  the fibration $g$ would be cohomologically non-detectable. This means that  the base would be birational to an abelian variety of dimension $<b(X)$, and therefore, by $(\ref{compo})$, the component $W$ would have dimension $<b(X)$. The last assertion follows  from the second equivalence of (a) via Lemma \ref{bijection}. This concludes the proof  of Step 1.

\vskip0.3truecm\noindent\textbf{Step 2. }  \emph{Conclusion of the proof of Theorem \ref{fibrations}. }  We make the following

\noindent \emph{Claim.   $b(X)>0$ if and only if $\dim V^i(X,\omega_X)\ge d-i$ for some $0<i<d$.  If $b(X)>0$ then $d-b(X)$ is the maximal index $i$ with  $0<i<d$ such that $\dim V^i(X,\omega_X)\ge d-i$.}
\proof  Concerning the first equivalence, if $b(X)>0$ then by (a) of Step (1) there is a weakly-$\chi$-positive fibration $g$ of base dimension $b(X)$ and therefore, by Lemma \ref{bijection}, there is a component of $V^{d-b(X)}(X,\omega_X)$ 
 of dimension $\ge b(X)$. The other implication follows from Remark \ref{bad}. The last assertion follows by the same reasons.\endproof
 Now let $X$ and $Y$ be derived-equivalent varieties. By the Claim the integers $b(X)$  and b(Y) are respectively determined by the dimensions of the various
loci $V^i(X,\omega_X)$ and $V^i(Y,\omega_Y)$. Therefore Corollary \ref{cor2} yields that $b(X)=b(Y):=b$. 
 From Step 1 cohomologically detectable fibrations of base dimension equal to $b$ are weakly $\chi$-positive and their equivalence classes correspond to all components of dimension $\ge b$ of $V^{d-b}(X,\omega_X)$ and such components are standard. Therefore by Lemma \ref{bijection} they are in $1-1$ correspondence with the corresponding subset of abelian subvarieties of $\Pic0 X$. The same for $Y$. Therefore by Corollary \ref{cor2} the Rouquier isomorphism induces a bijection between the sets of equivalence classes of cohomologically detectable fibrations of base dimension $b$ on $X$ and $Y$.

  It remains to prove that there is a bijection preserving, up to equivalence, the bases of the fibrations.\footnote{Our notion of equivalence of fibrations is weaker than Lombardi's notion of \emph{isomorphism of irrational pencils} (\cite{lombardi-fibrations}). However, as in Lombardi's paper, it can be proved that the bijection of Theorem \ref{fibrations} is  base-preserving not only up to equivalence, but also up to isomorphism of the bases of the Stein factorizations of the maps $f$ of (\ref{diagram}).} To begin with, we note that the above-constructed bijection is base-preserving on the subset of fibrations whose bases are birational to abelian varieties. Indeed Step 1 shows that they correspond to components of dimension $\ge b$ of $V^{d-b}(X,\omega_X)$ such that their parallel abelian varieties, namely ${g^\prime}^*\Pic0 S^\prime$, are not components of $V^{d-b}(X,\omega_X)$. The same for $Y$. The Rouquier isomorphism sends  isomorphically such components of $V^{d-b}(X,\omega_X)$ to components of $V^{d-b}(Y,\omega_Y)$, say ${h^\prime}^*\Pic0 R^\prime$, with the same property. Both $S^\prime$ and $R^\prime$ are birational to abelian varieties, and their Picard tori are isomorphic. Therefore $S^\prime$ is birational to $R^\prime$.  Concenring the remaining fibrations, namely those whose bases are not birational to abelian varieties, by Step 1(a) they are $\chi$-positive. Therefore Proposition \ref{chi-positive} applies.\footnote{Here we are not claiming that this bijection coincides with the one constructed above, namely the one induced by the Rouquier isomorphism. However this is true, but the proof of this fact requires some tools not in use in this paper.}

\providecommand{\bysame}{\leavevmode\hbox
to3em{\hrulefill}\thinspace}


\begin{thebibliography}{EMS}

\bibitem[A]{abuaf} R. Abuaf, {Homological Units}, Int. Math. Res. Not. IMRN (2017), no. 22, 6943--6960.

\bibitem[B]{beauville} A. Beauville, {Annulation du $H^1$ pour les fibr\'es en droites plats}, {\em Complex algebraic varieties}, LNM 1507, Springer-Verlag 1992, 1--15.

\bibitem[Ca]{calda} A. C\u{a}ld\u{a}raru, The Mukai pairing, I: The Hochschild structure, preprint arXiv:math/0308079.

\bibitem[Cas]{castel}G. Castelnuovo,  {Sulle superficie aventi il genere aritmetico negativo}, Rend. Circ. Mat. Palermo 20 (1905) 55-- 60 


\bibitem[CH1]{ch1} J.A. Chen, Ch. Hacon, {Characterization of abelian varieties}, Invent. Math. 143 (2001) 435--447 

\bibitem[CH2]{ch-iitaka1} J. A. Chen, Ch. Hacon, {Pluricanonical maps of varieties of maximal Albanese dimension}, Math. Ann. 320 (2001), 367--380.





\bibitem[CJ]{chen-jiang} J.A. Chen, Z. Jiang, {Positivity in varieties of maximal Albanese dimension},  J.
Reine Angew. Math. 736 (2018) 225--253,

\bibitem[DF]{defranchis} M. De Franchis, {Sulle superficie algebriche le quali contengono un fascio irrazionale di curve}, Rend. Circ. Mat. Palermo 20 (1905) 49--54 

\bibitem[EL]{el} L. Ein and R. Lazarsfeld, Singularities of theta divisors and the birational geometry of irregular varieties,
J. Amer. Math. Soc. 10 (1997) 243--258.



\bibitem[GL]{gl2}
M. Green and R. Lazarsfeld, {Higher obstructions to deforming cohomology groups of line bundles}, J. Amer. Math. Soc. 1 (1991)  87--103.

\bibitem[H]{hacon}
Ch. Hacon, {A derived category approach to generic vanishing}, J.
Reine Angew. Math. 575 (2004), 173--187.

\bibitem[HPS]{hps} Ch. Hacon, M. Popa, Ch Schnell, {Algebraic fibers spaces over abelian varieties: around a recent theorem by Cao and Paun}
to appear in the Contemporary Mathematics volume in honor of L. Ein's 60th birthday


\bibitem[Hu]{huybrechts} D. Huybrechts, {\em Fourier-Mukai transforms in algebraic geometry}, Clarendon Press - Oxford (2006)



\bibitem[K1]{kollar1}
J. Koll\'ar, {Higher direct images of dualizing sheaves I}, Ann. of
Math. 123 (1986), 11--42.



\bibitem[K2]{kollar2}
J. Koll\'ar, {Higher direct images of dualizing sheaves II}, Ann. of
Math. 124 (1986), 171--202.





\bibitem[LPo]{bgg} R. Lazarsfeld, M. Popa, {Derivative complex, BGG correspondence, and numerical inequalities for compact K\"ahler manifolds}, Invent. Math. 182 (2010)  605--633

\bibitem[LPoS]{lps} R. Lazarsfeld, M. Popa, Ch. Schnell {Canonical cohomology as an exterior module}, 
Pure and Applied Math. Quarterly  7 (2011)  1529--1542

\bibitem[Lo1]{lombardi-derived} L. Lombardi, {Derived invariants of irregular varieties and Hochschild homology,} Algebra Number
Theory 8 (2014), no. 3, 513--542.

\bibitem[Lo2]{lombardi-fibrations}L.Lombardi, {Derived equivalence and fibrations over curves and surfaces}, arXiv 1803.08656v1 [mathAG]




\bibitem[LoPo]{popa-conj2} L. Lombardi, M. Popa, {Derived equivalence and non-vanishing loci II}, 
{\em London Math. Soc. Lecture Note Series, 417, Recent Advances in Algebraic Geometry}, Cambridge University Press (2015) 291--306






\bibitem[O]{orlov} D. Orlov, {Derived categories of coherent sheaves and equivalences between them}, Russian Mathematical Surveys 58 (2003) 511--591



\bibitem[P1]{msri} G. Pareschi, {Basic results on irregular varieties via Fourier-Mukai methods}, in {\em Current developments in Algebraic Geometry}, MSRI Publications 59 (2011) 279--403.

\bibitem[P2]{standard} G. Pareschi, {Standard canonical support loci}, Rend. Circ. Mat. Palermo, II. Ser 66 (2017) 137--157











\bibitem[PPoS]{pps} G. Pareschi, M. Popa, Ch. Schnell, {Hodge modules on complex tori and generic vanishing for compact K\"ahler manifolds}, Geo. Top. 21 (2017), 2419--2460

\bibitem[Po]{popa-conj1} M. Popa, {Derived equivalence and non-vanishing loci},
{\em Clay Mathematical Proceedings} 18,  AMS 2013, 567-575

\bibitem[PoS1]{ps1} M. Popa and Ch. Schnell, Derived invariance of the number of holomorphic 1-forms and vector fields 
Ann. Sci. \'Ecole Norm. Sup. 44 (2011) 527--536

\bibitem[PoS2]{ps2} M. Popa and Ch. Schnell, {Generic vanishing theory via mixed Hodge modules}, 
Forum of Mathematics, Sigma 1 (2013) 1--60







\bibitem[R]{rouquier} R. Rouquier, {Automorphismes, graduations et cat\'egories triangul\'ees}, J. Inst. Math. Jussieu 10
(2011)  713--751.

\bibitem[Si]{simpson} C. T. Simpson, {Subspaces of moduli spaces of rank one local systems}, Ann. Sci. \'Ecole Norm. Sup. 26 (1993) 361--401

\bibitem[Sw]{swan} R. Swan, {Hochschild cohomology of quasi-projective schemes}, J. Pure and Applied Algebra 110 (1996) 57--80

\end{thebibliography}
\end{document}